\documentclass[a4paper,10pt,reqno]{amsart}

\usepackage[UKenglish]{babel}
\usepackage{amsmath}
\usepackage{amssymb}
\usepackage{amsthm}
\usepackage{verbatim}
\usepackage{stackrel}
\usepackage{dutchcal}

\usepackage[shortlabels]{enumitem}
\usepackage[T1]{fontenc}

\usepackage{comment,cite}	
\usepackage{hyperref}
\hypersetup{
    colorlinks=true,
    linkcolor=blue,
    citecolor=red,
    filecolor=magenta,      
    urlcolor=cyan,
    pdftitle={Domination of semigroups on non-commutative spaces},
    linktocpage=true,
}

\usepackage[utf8]{inputenc}

\usepackage{amssymb}
\usepackage{amsmath}
\usepackage{amsthm}
\usepackage{bbm}
\usepackage{verbatim, stackrel}

\usepackage[shortlabels]{enumitem}
\usepackage{marginnote}
\usepackage{color}

\usepackage{stmaryrd}
\usepackage{tikz}
\usepackage{tikz-cd}

\newcommand{\bbC}{\mathbb{C}}

\newcommand{\bbR}{\mathbb{R}}

\newcommand{\calH}{\mathcal{H}}

\newcommand{\calL}{\mathcal{L}}

\newcommand{\calU}{\mathcal{U}}
\newcommand{\calV}{\mathcal{V}}

\newcommand{\calb}{\mathcal b}
\newcommand{\fraka}{\mathfrak a}
\newcommand{\frakb}{\mathfrak b}
\newcommand{\frakc}{\mathfrak c}
\newcommand{\frakd}{\mathfrak d}
\newcommand{\frakM}{\mathfrak M}

\DeclareMathOperator{\re}{Re} % real part
\newcommand{\argument}{\mathord{\,\cdot\,}} % argument dot for functions (with correct spacing)
\DeclareMathOperator{\sgn}{sgn} % signum 
\newcommand{\norm}[1]{\left\lVert #1 \right\rVert} % norm
\newcommand{\modulus}[1]{\left\lvert #1 \right\rvert} % modulus
\newcommand{\duality}[2]{\left\langle#1\, ,\, #2\right\rangle} % duality / scalar product
\newcommand{\dom}[1]{\operatorname{dom}\left(#1\right)} % domain of an operator
\newcommand{\Proj}{\operatorname{Proj}} % projection

\theoremstyle{definition}
\newtheorem{definition}{Definition}[section]

\newtheorem*{remark*}{Remark}
\newtheorem*{remarks*}{Remarks}
\newtheorem{example}[definition]{Example}

\theoremstyle{plain}
\newtheorem{proposition}[definition]{Proposition}

\newtheorem{theorem}[definition]{Theorem}

\numberwithin{equation}{section} % enumerate formulas within sections

\begin{document}

\title[Domination of semigroups on non-commutative spaces]{Domination of semigroups on standard forms of von Neumann algebras}
\author{Sahiba Arora}
\address{Sahiba Arora, Department of Applied Mathematics, University of Twente, 217, 7500 AE, Enschede, The Netherlands}
\email{s.arora-1@utwente.nl}
\author{Ralph Chill}
\address{Ralph Chill, Technische Universität Dresden, Institut für Analysis, Fakultät für Mathematik , 01062 Dresden, Germany}
\email{ralph.chill@tu-dresden.de}
\author{Sachi Srivastava}
\address{Sachi Srivastava, Department of Mathematics, University of Delhi, South Campus, New Delhi-21, India}
\email{ssrivastava@maths.du.ac.in}
\thanks{This project is supported by the VAJRA scheme VJR/2018/000127, of the Science and Engineering Research Board, Department of Science and Technology,  Govt. of India. The second and third authors gratefully acknowledge the support from VAJRA}

\date{\today}
\begin{abstract}
    Consider $(T_t)_{t\ge 0}$ and $(S_t)_{t\ge 0}$ as real $C_0$-semigroups generated by closed and symmetric sesquilinear forms on a standard form of a von Neumann algebra. We provide a characterisation for the domination of the semigroup $(T_t)_{t\ge 0}$ by $(S_t)_{t\ge 0}$, which means that $-S_t v\le T_t u\le S_t v$ holds for all $t\ge 0$ and all real $u$ and $v$ that satisfy $-v\le u\le v$. This characterisation extends the Ouhabaz characterisation for semigroup domination to the non-commutative  $L^2$-spaces. Additionally, we present a simpler characterisation when both semigroups are positive as well as consider the setting in which $(T_t)_{t\ge 0}$ need not be real.
\end{abstract}

\maketitle

\section{Introduction}

Domination of $C_0$-semigroups was first investigated by Simon \cite{Si77} inspired by the concept of domination of operators that occurred in the form of Kato's inequality \cite{Ka72}. Since then, the notion of domination has garnered considerable attention in the theory of $C_0$-semigroups. 

Whereas characterisations in terms of the generators were known much earlier (see, for instance,  \cite{HeScUh77, Si79}, and \cite[Section~C-II-4]{Na86}), a characterisation in terms of the generating forms was given by Ouhabaz in \cite{Ou96} and  by Barth\'elemy in \cite{By96} and was generalised to semigroups  acting on two different $L^2$-spaces in \cite{MaVoVo05}. These results were recently considered within an abstract framework on more general ordered Hilbert spaces in \cite{LeScWi20}.

In this article, we generalise Ouhabaz's result to the non-commutative setting. Following Cipriani \cite{Ci97, Ci08} -- who generalised the Beurling-Deny criterion to the non-commutative setting (see \cite[Theorem~2.53]{Ci08} or \cite[Proposition~4.5 and Theorem~4.7]{Ci97}) -- we consider our state space to be a standard form of a von Neumann algebra. This has the advantage that for real elements, there is a concept of positive and negative parts (recalled below). 

It turns out that when generalising the known domination results from the commutative setting to the non-commutative setting one has to avoid the modulus as much as possible, both in the definition of domination and in the proofs. This is due to the well-known fact that in the non-commutative setting, the modulus no longer satisfies the triangle inequality so some convex structure is lost. Even though we define the modulus below, we only make use of it in the commutative setting. 

In the remainder of this section, we fix our notations about sesquilinear forms,
briefly recall the notion of standard forms, and introduce the concept of generalised ideals. Subsequently, in Section~\ref{sec:domination-general}, we state and prove our main result about the domination of real semigroups. Then, in Section~\ref{sec:domination-both-positive}, we restrict ourselves to the particular case when both semigroups are positive. Finally, in Section~\ref{sec:domination-not-real}, we drop the assumption for the dominated semigroup to be real.

\subsection*{Sesquilinear forms on Hilbert spaces and their associated semigroups} 

Throughout the article, we consider $C_0$-semigroups on complex Hilbert spaces, the generators of which arise from sesquilinear forms. Let $\calH$ be a complex Hilbert space, let $\dom{\fraka}\subseteq\calH$ be a subspace and let $\fraka : \dom{\fraka} \times \dom{\fraka} \to \bbC$ be a sesquilinear form. The subspace $\dom{\fraka}$ is called the \emph{domain} of the form $\fraka$. We say that $\fraka$ is 
\begin{itemize}
    \item \emph{densely defined} if $\dom{\fraka}$ is a dense subspace of $\calH$,
    \item \emph{symmetric} if $\fraka (u,v) = \overline{\fraka (v,u)}$ for every $u$, $v\in\dom{\fraka}$,
    \item \emph{accretive} if $\re\fraka (u,u) \geq 0$ for every $u\in\dom{\fraka}$, and
    \item \emph{closed} if there exists $\omega\in\bbR$ such that $\| u\|_{\dom{\fraka}}^2 := \re\fraka (u,u) + \omega \| u\|_{\calH}^2$ defines a complete norm on $\dom{\fraka}$.
\end{itemize}
We recall that every closed and symmetric form is continuous.
In particular, if the sesquilinear form $\fraka$ is densely defined, closed, and symmetric, then the operator $A$ on $\calH$ given by 
\begin{align*}
    \dom{A} & := \{ u\in \dom{\fraka} : \exists f\in\calH\ \forall y\in\dom{\fraka} : \fraka (u,v) = \duality{f}{v} \} , \\
    Au & := f ,
\end{align*}
is self-adjoint, bounded from below, and therefore the negative generator of an analytic $C_0$-semigroup of self-adjoint operators. In short, we simply say that the sesquilinear form generates the semigroup. If the form $\fraka$ is, in addition, accretive, then the semigroup is contractive. As we are interested in the domination of semigroups, and since domination is invariant under simultaneous scaling of both semigroups, we restrict ourselves to accretive forms.

Throughout, we deliberately identify symmetric sesquilinear forms with their quadratic counterparts $\fraka : \dom{\fraka} \to \bbR$, $\fraka (u) := \fraka (u,u)$, since one can reconstruct the symmetric sesquilinear form from the quadratic form via the polarisation identity. As is customary in the literature, we denote the sesquilinear form and the quadratic form by the same letter. We further identify the quadratic form (and in turn the sesquilinear form) with its extension $\fraka : \calH \to \bbR\cup \{ +\infty\}$, where
\[
\fraka (u) := \begin{cases}
    \fraka (u) & \text{if } u\in\dom{\fraka} , \\
    +\infty & \text{if } u\in\calH \setminus \dom{\fraka} .
\end{cases}
\]
This extended quadratic form renders some statements easier to formulate. Note that if the sesquilinear form is symmetric and accretive, then the quadratic form (extended quadratic form) takes values in $[0,\infty )$ (respectively, $[0,\infty ]$).

Frequently, we require the following theorem about the invariance of closed, convex sets in Hilbert spaces. The equivalence of (i)-(iv) is well-known (see, for example, \cite[Theorems~2.2 and~2.3]{Ou04}) and the equivalences of (i), (v), and (vi) follow from \cite[Th\'eor\`eme~1.9]{By96}.

\begin{theorem}
    \label{thm:barthelemy}
Let $(R_t)_{t\geq 0}$ be a $C_0$-semigroup on a Hilbert space $\calH$ which is generated by a densely defined, closed, symmetric, and accretive sesquilinear form $\frakc : \dom{\frakc} \times\dom{\frakc} \to\bbC$. 
Let $C \subseteq\calH$ be a closed, convex subset and let $P$ be the orthogonal projection onto $C$. The following are equivalent:
\begin{enumerate}[\upshape (i)]
    \item $R_tC\subseteq C$ for every $t\geq 0$.
    \item $\frakc (Pu) \leq \frakc (u)$ for every $u\in\calH$.
    \item $P(\dom{\frakc})\subseteq \dom{\frakc}$ and  $\re\frakc( u, u-Pu )\ge 0$ for all $u\in\dom{\frakc}$.
    \item $P(\dom{\frakc})\subseteq \dom{\frakc}$ and  $\re\frakc( Pu, u-Pu )\ge 0$ for all $u\in\dom{\frakc}$.
\end{enumerate}
If $C_0\subseteq\calH$ is another closed, convex subset such that $C\subseteq C_0$ and $R_tC_0\subseteq C_0$ for every $t\geq 0$, then~{\upshape (i)}-{\upshape (iv)} are also equivalent to:
\begin{enumerate}[resume]
    \item $\frakc (Pu) \leq \frakc (u)$ for every $u\in C_0$.
    \item $P(\dom{\frakc} \cap C_0)\subseteq \dom{\frakc}$ and  $\re\frakc( u, u-Pu )\ge 0$ for all $u\in\dom{\frakc}\cap C_0$.
\end{enumerate}
\end{theorem}

\subsection*{Standard forms of von Neumann algebras}

Let $\calH$ be a complex Hilbert space, and let $\calH_+\subseteq\calH$ be a positive cone, that is, for all $u,v\in\calH_+$ and all $\lambda\geq 0$, we have $u+\lambda v\in\calH_+$ and $\calH_+ \cap (-\calH_+) = \{ 0\}$. Elements of $\calH_+$ are called the \emph{positive} elements of $\calH$. %We assume that the cone $\calH_+$ is \emph{generating}, that is, the smallest complex subspace of $\calH$ containing $\calH_+$ is the space $\calH$ itself. 
Throughout, the cone $\calH_+$ is \emph{self-polar}, that is, 
\[
    \calH_+ = \{ u\in \calH : \duality{u}{v} \ge 0 \text{ for all }v\in \calH_+\} .
\]
Self-polarity of the cone ensures the following decomposition \cite[Proposition~2.3]{Ci08}: $\calH$ is the complexification of the real Hilbert space
\[
    \calH^J := \{ u \in \calH: \duality{u}{v} \in \bbR \text{ for all }v\in \calH_+\} , % = \calH_+ - \calH_+ , 
\] 
that is, $\calH = \calH^J + i \calH^J.$ Elements of $\calH^J$ are called \emph{real}. 
For a subspace $\calU$ of the Hilbert space $\calH$, we use the notation $\calU^J:= \calU \cap \calH^J$ for simplicity.
A self-polar cone $\calH_+$ turns $\calH^J$ into a real ordered space with order denoted by $\leq$, where for $u,v \in \calH^J, $ we have $ u \le v  $ if $ 
 v - u  \in \calH_+$.  It also gives rise to an anti-unitary involution
\[
    J: \calH \to \calH, \qquad (u+iv)\mapsto (u-iv) \text{ for } u,v \in \calH^J.
\]

Let $u_+ :=\Proj(u,\calH_+)$  (respectively, $u_- := \Proj (-u,\calH_+)$) denote the orthogonal projection of a vector $u\in\calH^J$ (respectively, $-u$) onto the closed convex set $\calH_+$.  For every $u\in \calH^J$, we have $u_+$, $u_-\in \calH_+$ and  $u=u_+ - u_-$ (Jordan decomposition) with $\duality{u_+}{u_-}=0$. This decomposition is unique.  Similarly, define
\[
    u \vee v := \Proj(u,v+\calH_+) \quad \text{and} \quad u \wedge v:= \Proj(u,v-\calH_+)
\]
for $u,v\in \calH^J$. This notation is reminiscent of the vector lattice case and is justified by \cite[Lemma~4.4(i)]{Ci97}, which says that if $\sup(u,v)$ exists (to be understood in the order sense) then $\sup(u,v)=u \vee v$ and if $\inf(u,v)$ exists (in the order sense), then $\inf(u,v)=u\wedge v$. For this reason, the vectors $u_+$ and $u_-$ are respectively called the \emph{positive} and \emph{negative} parts of $u$.
In fact, we even know from \cite[Lemma~4.4(iii)]{Ci97} that
\[
    u \vee v= u+ (v-u)_+ = v+ (u-v)_+.
\]
In particular, $u \vee v$ is, at least, an upper bound of both $u$ and $v$. Analogously, $u \wedge v$ is a lower bound of both $u$ and $v$.
Moreover, just like in the vector lattice case, the above projections satisfy the standard properties.
For instance, $u\vee 0=u_+$, $u\wedge 0=-u_-$, and $u \vee v+ u\wedge v=u+v$ for all $u,v\in \calH^J$. A proof of these identities can be found in \cite[Lemma~4.4]{Ci97}. Our arguments freely make use of the properties listed in \cite[Lemma~4.4]{Ci97} without citing.

Throughout, we let $(\frakM, \calH, \calH_+,J)$ be a \emph{standard form}, that is, a quadruple consisting of a von Neumann algebra $\frakM$ acting on a complex Hilbert space $\calH$ (in particular, $\frakM\subseteq \calL(\calH))$ endowed with a self-polar cone $\calH_+$ and an anti-linear involution $J : \calH \to \calH$ such that 
\begin{enumerate}[(a)]
    \item $J\frakM J=\frakM'$, where $\frakM'$ is the commutant of $\frakM$ in $\calL (\calH )$,
    \item $JaJ= a^*$ for all $a \in \frakM \cap \frakM'$,
    \item $J u=u$ for all $u \in \calH_+$, and 
    \item $aJaJ(\calH_+)\subseteq \calH_+$ for all $a\in \frakM$.
\end{enumerate}

Positive cones arising from standard forms of von Neumann algebras have been characterised by Connes \cite{Co74}. For every von Neumann algebra, there is a standard form, and if $(\frakM, \calH, \calH_+,J)$ and $(\hat{\frakM}, \hat{\calH}, \hat{\calH}_+,\hat{J})$ are two standard forms such that the von Neumann algebras $\frakM$ and $\hat{\frakM}$ are isomorphic via an isomorphism, say $\Phi$, then there exists a unitary operator $U: \calH \to \hat{\calH}$ such that $\Phi (\frakM) = U\frakM U^{-1}$, $\hat{J} = UJU^{-1}$, and $U (\calH_+) = \hat{\calH}_+$ (see \cite[Theorem IX.1.4]{Ta03II} %, \cite[Theorem 2.3]{Ha75}, 
or \cite[Proposition~2.16]{Ci08}). In this sense, the standard form is uniquely determined by the von Neumann algebra and is usually denoted by $(\frakM, L^2(\frakM), L_+^2(\frakM), J)$. The notation $\calH=L^2(\frakM)$ is reminiscent of the case when the von Neumann algebra $\frakM$ has a semi-finite, normal, and faithful trace $\tau$ because then $L^2 (\frakM )$ is the associated non-commutative $L^2$-space which is, by definition, the completion of the space $\{ a\in\frakM : \tau (aa^*) <\infty \}$ with respect to the inner product $\duality{a}{b} := \tau (ab^*)$.

An operator $T: \calH \to \calH $ is said to be \emph{real} if $ T (\calH^J) \subseteq \calH^J $ and \emph{positive} if $ T (\calH_{+}) \subseteq  \calH_{+} $. Similarly, a semigroup $(T_t)_{t\geq 0}$ is \emph{real} (respectively, \emph{positive}) if the operators $T_t$ are real (respectively, positive) for every $t\geq 0$. 

Lastly, if $(T_t)_{t\geq 0}$ is a real $C_0$-semigroup generated by a sesquilinear form $\fraka$, then $\dom{\fraka}^J$ is generating for $\dom{\fraka}$ and $\fraka(u,v)\in \bbR$ whenever $u,v\in \dom{\fraka}^J$. Indeed, this follows from Theorem~\ref{thm:barthelemy}; the proof actually follows exactly as the proof of \cite[Proposition~2.5]{Ou04} taking $C=\calH^J$.

\subsection*{Generalised ideals}

Let $(\Omega,\mu)$ be a $\sigma$-finite measure space and $\calH:=L^2(\Omega,\mu)$. Let $(T_t)_{t\ge 0}$ and $(S_t)_{t\ge 0}$ be real and self-adjoint $C_0$-semigroups on the Hilbert space $\calH$ generated by densely defined, closed, and symmetric sesquilinear forms $\fraka: \dom{\fraka} \times \dom{\fraka} \to \bbC$ and $\frakb: \dom{\frakb} \times \dom{\frakb} \to \bbC$ respectively. If $(S_t)_{t\ge 0}$ dominates $(T_t)_{t\ge 0}$, that is, 
\[
    \modulus{T_t u}\le S_t\modulus{u}
\]
for all $u\in \calH$, then $\dom{\frakb}$ is a  sublattice of $\calH$  \cite[Proposition~2.20]{Ou04} and $\dom{\fraka}$ is a so-called \emph{generalised ideal} of $\dom{\frakb}$; see \cite[Theorem 4.1]{MaVoVo05} or \cite[Theorem~2.21]{Ou04}. We point out that generalised ideals were simply called \emph{ideals} in \cite{Ou04}. However, we refrain from using that terminology as, in general, it does not coincide with the notion of a \emph{lattice ideal}.

In the non-commutative setting, when $(\frakM, \calH, \calH_+,J)$ is a standard form of a von Neumann algebra $\frakM$, and when $\calU$ and $\calV$ are two subspaces of the Hilbert space $\calH$, then we will say that the real space $\calU^J$ is a \emph{generalised ideal} of the real space $\calV^J$ if both the conditions
\begin{equation}\label{cond:one-generalised ideal}
        (u-v)_+ - (u+v)_- \in \calU
\end{equation} 
and
\begin{equation} \label{cond:two-generalised ideal}
    (u-v)_+ + (u+v)_- \in \calV
\end{equation}
hold whenever $u\in \calU^J$ and $v\in\calV^J$.

Note that if $\calU^J$ is a generalised ideal of $\calV^J$, then condition \eqref{cond:two-generalised ideal} above implies that
\begin{equation}
    \label{eq:ideal-modulus}
    u\in \calU^J \Rightarrow \modulus{u} := u_+ + u_- \in \calV.
\end{equation}

A rewriting of the  terms in~\eqref{cond:one-generalised ideal} and~\eqref{cond:two-generalised ideal} yields the following alternate definition of generalised ideals.

\begin{proposition}
    \label{prop:prop-ideal-equivalence}
    Let $(\frakM, \calH, \calH_+,J)$ be a standard form of the von Neumann algebra $\frakM$ and let $\calU$ and $\calV$ be two subspaces of $\calH$. 
    
    Then $\calU^J$ is a generalised ideal of $\calV^J$ if and only if
    \[ 
    (u+v)_+ - (v-u)_+ \in \calU \quad \text{
        and} \quad  (u+v)_+ + (v-u)_+ \in \calV 
    \]
    whenever $u\in \calU^J$ and $v\in \calV^J$.
\end{proposition}

\begin{proof}
For $u$, $v\in \calH^J$, we observe that 
    \begin{equation}\label{eq:ideal-equivalence-1}
        \begin{aligned}
            (u+v)_+ - (v-u)_+ &= u + v +(u+v)_- -\big(v-u + (u-v)_+\big)\\
                              &= 2u - \big((u-v)_+ - (u+v)_-\big)
        \end{aligned}
    \end{equation}
    and
    \begin{equation}\label{eq:ideal-equivalence-2}
        \begin{aligned}
            (u+v)_+ + (v-u)_+ &= u + v +(u+v)_- +\big(v-u + (u-v)_+\big)\\
                              &= 2v + \big((u-v)_+ + (u+v)_-\big);
        \end{aligned}
    \end{equation}
from which the assertion is immediate.
\end{proof}

%We end this section by showing that Definition~\ref{def:generalized-ideal} coincides with the usual definition of a  \emph{generalised ideal} in \cite{MaVoVo05}, (that is, when conditions \ref{gi1} and \ref{gi2} stated above hold)  provided the form $\frakb$ generates a positive semigroup (!!). We point out that these objects were simply called \emph{ideals} in \cite{Ou04}. However, we refrain from using this terminology as, in general, it does not coincide with the notion of a \emph{lattice ideal}.

We end this section by showing that in the commutative setting, if $\calV$ is a sublattice, then $\calU^J$ is a generalised ideal of $\calV^J$ in the sense of our definition if and only if $\calU^J$ is a generalised ideal of $\calV^J$ in the sense of \cite[Definition~3.3]{MaVoVo05}. Keep in mind that this is only a statement about the real parts of the spaces $\calU$ and $\calV$, and not about the full spaces themselves.

\begin{proposition}
    \label{prop:ideal-commutative}
    Let $(\Omega,\mu)$ be a $\sigma$-finite measure space, let $\calH:=L^2(\Omega,\mu)$, and let $\calU$ and $\calV$ be two subspaces of $\calH$. 
    
    If $\calV$ is a sublattice, then $\calU^J$ is a generalised ideal of $\calV^J$ in the sense of our definition if and only if 
    \begin{enumerate}[\upshape (a)]
        \item whenever $u\in \calU^J$, then $\modulus{u}\in \calV$, and
        \item if $u \in \calU^J, v\in \calV^J$, and $\modulus{v}\le \modulus{u}$,  then $v\sgn u \in \calU$,
    \end{enumerate}
    that is, $\calU^J$ is a generalised ideal of $\calV^J$ in the sense of \cite[Definition 3.3]{MaVoVo05}. 
\end{proposition}

\begin{proof}
    For real $u,v\in \calH$, we set
    \[
        \tilde u:= \frac{1}{2} (u+v)_+ -\frac{1}{2} (v-u)_+ \quad \text{and}\quad \tilde v=\frac{1}{2} (u+v)_+ +\frac{1}{2} (v-u)_+.
    \]
    
    The commutative setting allows us to rewrite
    \begin{align*}
        \big( \tilde u , \tilde v \big) 
        &= \begin{cases}
                    (u,v),\quad &\text{if }\modulus{u}\le v\\
                    (0,0),\quad  &\text{if }\modulus{u}\le -v\\
                    \frac{1}{2}(u+v,u+v),\quad  &\text{if }\modulus{v}\le u\\
                    \frac{1}{2}(u-v,v-u),\quad  &\text{if }\modulus{v}\le -u
            \end{cases} \\
        &= \frac{1}{2}\big((\modulus{u}+ \modulus{u}\wedge v)_+\sgn u, (v+ \modulus{u}\vee v)_+  \big);
    \end{align*}
    which is exactly the projection $P(u,v)$ onto the set
    \[
         \{(a,b)\in \calH \times \calH : \modulus{a}\le b\}
    \]
    for real $u,v\in \calH$.

By assumption, $\calV$ is a sublattice of $\calH$. Therefore, by \cite[Propositon~3.5]{MaVoVo05}, the projection $P$ leaves $\calU^J\times \calV^J$ invariant if and only if both~(a) and~(b) hold. With the aid of Proposition~\ref{prop:prop-ideal-equivalence}, it follows that $\calU^J$ is a generalised ideal of $\calV^J$ if and only if both~(a) and~(b) are true.
\end{proof}

\section{Domination of a real semigroup by a positive semigroup}
    \label{sec:domination-general}

Let $(\frakM, \calH, \calH_+,J)$ be a standard form of the von Neumann algebra $\frakM$ and let $(T_t)_{t\ge 0}$ and $(S_t)_{t\ge 0}$ be real $C_0$-semigroups on the Hilbert space $\calH$. We shall say that the semigroup $(T_t)_{t\ge 0}$ is \textit{dominated}  by $(S_t)_{t\ge 0}$ if 
\[
    \left( -v \le u \le v \right) \Rightarrow \left(- S_t v \le T_t u\le S_t v \right)
\]
for all $u,v\in \calH^J$ and all $t\ge 0$. 

In the commutative case, that is, if $\calH$ is a Hilbert lattice, each $u\in \calH^J$ satisfies $\modulus{u}=\sup(-u,u)$. Therefore, in this case, the above definition of domination is equivalent to the usual definition of domination \cite[Section~2.3]{Ou04}, which is,
\[
     \modulus{T_t u}\le S_t \modulus{u} \quad (u\in \calH, t\ge 0).
\]
In this section, we characterise the domination of $(T_t)_{t\ge 0}$ by $(S_t)_{t\ge 0}$ generated by sesquilinear forms $\fraka$ and $\frakb$ in terms of $\fraka$ and $\frakb$. Recall that we deliberately identify symmetric sesquilinear forms with extended quadratic forms.

\begin{theorem}
    \label{thm:general}
    Let $(\frakM, \calH, \calH_+,J)$ be a standard form of the von Neumann algebra $\frakM$ and let $(T_t)_{t\ge 0}$ and $(S_t)_{t\ge 0}$ be real and self-adjoint $C_0$-semigroups on the Hilbert space $\calH$ generated by densely defined, closed, symmetric, and accretive sesquilinear forms $\fraka: \dom{\fraka} \times \dom{\fraka} \to \bbC$ and $\frakb: \dom{\frakb} \times \dom{\frakb} \to \bbC$ respectively. 

    Assume that the semigroup $(S_t)_{t\ge 0}$ is positive. Then the following conditions are equivalent.
    \begin{enumerate}[\upshape(i)]
        \item The semigroup $(T_t)_{t\ge 0}$ is dominated by $(S_t)_{t\ge 0}$.

        \item The inequality $\fraka(u-\hat u) + \frakb(v+\hat v)\le \fraka(u)+\frakb(v)$ holds for all $u,v\in \calH^J$.

        \item The inequality $\fraka(\tilde u) + \frakb(\tilde v)\le \fraka(u)+\frakb(v)$ holds for all $u,v\in \calH^J$.

        \item The subspace $\dom{\fraka}^J$ is a generalised ideal of $\dom{\frakb}^J$ and the inequality $\fraka(u, \hat u) \ge \frakb(v, \hat v)$ holds for all $(u,v)\in \dom{\fraka}^J\times\dom{\frakb}^J$.

        \item The subspace $\dom{\fraka}^J$ is a generalised ideal of $\dom{\frakb}^J$ and the inequality $\fraka(u,\tilde u) +\frakb(v,\tilde v) \le \fraka(u)+\frakb(v)$ holds for all $(u,v)\in \dom{\fraka}^J\times\dom{\frakb}^J$.

        \item The subspace $\dom{\fraka}^J$ is a generalised ideal of $\dom{\frakb}^J$ and the inequality $\fraka(\hat u)+\frakb(\hat v) \le \fraka(u, \hat u) - \frakb(v, \hat v)$ holds for all $(u,v)\in \dom{\fraka}^J\times\dom{\frakb}^J$.

        \item The subspace $\dom{\fraka}^J$ is a generalised ideal of $\dom{\frakb}^J$ and the inequality $\fraka(\tilde u)+\frakb(\tilde u)\le \fraka(u,\tilde u)+\frakb(v,\tilde v)$ holds for all $(u,v)\in \dom{\fraka}^J\times\dom{\frakb}^J$.
    \end{enumerate}
    Here, we have used the notations,
    \[
        \hat u:= \frac{1}{2} (u-v)_+ -\frac{1}{2} (u+v)_- \quad \text{and}\quad \hat v:=\frac{1}{2} (u-v)_+ +\frac{1}{2} (u+v)_-
    \]
    and
    \[
        \tilde u:= \frac{1}{2} (u+v)_+ -\frac{1}{2} (v-u)_+ \quad \text{and}\quad \tilde v:=\frac{1}{2} (u+v)_+ +\frac{1}{2} (v-u)_+
    \]
    for $u,v\in\calH^J$.
\end{theorem}

We point out that if $(u,v)\in \dom{\fraka}^J\times\dom{\frakb}^J$, then the inequality in condition~(iii) of Theorem~\ref{thm:general} implies that $(\tilde u,\tilde v)\in \dom{\fraka}\times\dom{\frakb}$. In particular, $\dom{\fraka}^J$ is a generalised ideal of $\dom{\frakb}^J$ by Proposition~\ref{prop:prop-ideal-equivalence}. The same conclusion holds for the inequality in condition~(ii).

\begin{proof}[Proof of Theorem~\ref{thm:general}]
    We start by considering the closed convex set
    \[
        C := \{ (a,b)\in \calH \times \calH:  -b\le a\le b\}
    \]
    and the product semigroup on $\calH\times \calH$ given by
    \[
        R_t:= \begin{pmatrix}
                    T_t & 0     \\
                    0    & S_t  
              \end{pmatrix} \qquad (t\ge 0).
    \]
    Since the semigroup $(R_t)_{t\ge 0}$ is real, the domination condition~(i) is equivalent to the closed convex set $C$ being invariant under the semigroup $(R_t)_{t\ge 0}$,  that is,  $R_t C \subseteq C$ for all $t\ge 0$. This allows us to make use of Theorem~\ref{thm:barthelemy}. 
    For this purpose, we show that the orthogonal projection $P$ of $\calH\times \calH$ onto $C$ satisfies
    \begin{equation}\label{eq:projection-general}
        P(u,v)=(u-\hat u, v +\hat v)= (\tilde u,\tilde v)\quad (u,v\in \calH^J).
    \end{equation}
    Actually, the second equality is a direct consequence of~\eqref{eq:ideal-equivalence-1} and~\eqref{eq:ideal-equivalence-2}.
    
    In order to show the first equality in~\eqref{eq:projection-general}, we make use of the following equivalence:~$(u',v')=P(u,v)$ if and only if 
    \begin{equation}
        \label{eq:projection-criterion}
        (u',v')\in C \text{ and } \re\duality{(u,v)-(u',v')}{(a,b)-(u',v')}\le 0 \text{ for all }(a,b)\in C.
    \end{equation}
    So, fix $u,v\in \calH^J$. The inequalities
    \begin{align*}
        v+\hat v-(u-\hat u)  = v - u + (u-v)_+
                             = (u-v)_- 
                            \ge 0 
    \end{align*}
    and
    \begin{align*}
        v+\hat v+(u-\hat u)  = u + v + (u+v)_-
                             = (u+v)_+
                             \ge 0 
    \end{align*}
    tell us that $(u-\hat u,v+\hat v)\in C$. In addition, for $(a,b)\in C$, we have
    \begin{align*}
        4A &:= 4\duality{(u,v)-(u-\hat u,v+\hat v)}{(a,b)-(u-\hat u,v+\hat v)}\\
          & = 2\duality{\big((u-v)_+ -(u+v)_-, - (u-v)_+ -(u+v)_- \big) }{ (a+\hat u,b-\hat v) }\\
          & = 2\duality{(u-v)_+}{a-b} -2\duality{(u+v)_-}{a+b}\\
          &\qquad\qquad+\duality{(u-v)_+}{(u-v)_-}  +\duality{(u+v)_-}{(u+v)_+}
    \end{align*}
    Using orthogonality of the positive and the negative part, we deduce that
    \[
        2A = \duality{(u-v)_+}{a-b} -\duality{(u+v)_-}{a+b} \le 0.
    \]
    We have thus verified that both the conditions in~\eqref{eq:projection-criterion} are fulfilled. Consequently, the first equality in~\eqref{eq:projection-general} is true.

    Next, note that the semigroup $(R_t)_{t\ge 0}$ on $\calH\times \calH$ is generated by the form $\frakc$ with domain  $\dom{\frakc}=\dom{\fraka}\times \dom{\frakb}$ and
    \[
        \frakc( (u_0,v_0), (u_1, v_1) ) = \fraka(u_0, u_1)+ \frakb(v_0,v_1).
    \]
    As the semigroup is real and the form is symmetric, so the semigroup $(R_t)_{t\ge 0}$ leaves $C$ invariant if and only if $\frakc(P(u,v))\le \frakc( (u,v))$ for all $u,v\in \calH^J$ (Theorem~\ref{thm:barthelemy}). The equivalences ``(i) $\Leftrightarrow$ (ii)'' and ``(i) $\Leftrightarrow$ (iii)'' can now be deduced by substituting~\eqref{eq:projection-general}.

    Next, recalling that the semigroup $(R_t)_{t\ge 0}$ is real, it follows from~\eqref{eq:projection-general} that $P$ leaves the subspace $\dom{\frakc}$ invariant if and only if
    \[
         (\hat u,\hat v) \in \dom{\fraka}\times \dom{\frakb} \quad 
         \text{for all }(u,v)\in \dom{\fraka}^J\times\dom{\frakb}^J 
    \]
    or equivalently 
    \[
         (\tilde u,\tilde v) \in \dom{\fraka}\times\dom{\frakb} \quad 
         \text{for all }(u,v)\in \dom{\fraka}^J\times\dom{\frakb}^J. 
    \]
    In other words, $P$ leaves $\dom{\frakc}$ invariant if and only if $\dom{\fraka}^J$ is a generalised ideal of $\dom{\frakb}^J$; see Proposition~\ref{prop:prop-ideal-equivalence}.
    
    Now, fix $(u,v)\in \dom{\fraka}^J\times\dom{\frakb}^J$. Once again employing~\eqref{eq:projection-general}, we obtain that
    \begin{align*}
        \frakc( (u,v), (u,v)-P(u,v) ) & = \fraka(u, \hat u) - \frakb(v, \hat v)\\
                                      & = \fraka(u)- \fraka(u,\tilde u) +\frakb(v)- \frakb(v,\tilde v)  
    \end{align*}
    and
    \begin{align*}
        \frakc( P(u,v), (u,v)-P(u,v) ) & =  \fraka(u,\hat u)-\fraka(\hat u) - \frakb(v,\hat v)-\frakb(\hat v)\\
                                      & = \fraka(u,\tilde u)-\fraka(\tilde u)+\frakb(v,\tilde  v)-\frakb(\tilde v).
    \end{align*}
    In particular, 
    \begin{align*}
        \frakc( (u,v), (u,v)-P(u,v) )\ge 0 &\Leftrightarrow \fraka(u, \hat u) \ge \frakb(v, \hat v) \\
                                           & \Leftrightarrow \fraka(u,\tilde u) +\frakb(v,\tilde v) \le \fraka(u)+\frakb(v)
    \end{align*}
    and
    \begin{align*}
        \frakc( (Pu,v), (u,v)-P(u,v) )\ge 0  &\Leftrightarrow  \fraka(\hat u)+\frakb(\hat v) \le \fraka(u, \hat u) - \frakb(v, \hat v)\\
                                            & \Leftrightarrow \fraka(\tilde u)+\frakb(\tilde u)\le \fraka(u,\tilde u)+\frakb(v,\tilde v).
    \end{align*}
    We are now in a position to once again employ the characterisation of semigroup-invariance of a closed convex set Theorem~\ref{thm:barthelemy} which gives the equivalence of~(i) with each of~(iv)-(vii).
\end{proof}

Note that the results in \cite{LeScWi20} are actually considered in the general setting of ordered Hilbert spaces with self-polar cones as well. However, their definition of domination of semigroups requires the existence of \emph{absolute pairings} (see, for instance, \cite[Definition~1.24]{LeScWi20}) which cannot be guaranteed in the non-commutative setting. In particular, the results of \cite{LeScWi20} are inapplicable in our setting.
In fact, the implication \mbox{``(iii) $\Rightarrow$ (i)''} in \cite[Theorem~3.5]{LeScWi20} is only true for \emph{isotone projection cones}, that is, in the Hilbert lattice setting; see \cite[Theorem~1.15]{LeScWi20}.

\section{Domination of positive semigroups}
    \label{sec:domination-both-positive}

Let $(\frakM, \calH, \calH_+,J)$ be a standard form of the von Neumann algebra $\frakM$ and let $(T_t)_{t\ge 0}$ and $(S_t)_{t\ge 0}$ be self-adjoint $C_0$-semigroups on the Hilbert space $\calH$ generated by densely defined, closed, symmetric, and accretive sesquilinear forms $\fraka: \dom{\fraka} \times \dom{\fraka} \to \bbC$ and $\frakb: \dom{\frakb} \times \dom{\frakb} \to \bbC$ respectively. 

In this section, we show that the situation of Theorem~\ref{thm:general} becomes simpler if both the semigroups $(T_t)_{t\ge 0}$ and $(S_t)_{t\ge 0}$ are positive. 
 In particular, our characterisation is the same as in the commutative case \cite[Theorem~2.24 and Proposition~2.23]{Ou04}.

It is easy to see that due to the positivity of both semigroups, the definition -- stated in Section~\ref{sec:domination-general} -- of domination of $(T_t)_{t\ge 0}$ by $(S_t)_{t\ge 0}$ simplifies to 
\[
    T_t u\le S_t u \quad (u \in \calH_+, t\ge 0).
\]

\begin{theorem}
    \label{thm:both-positive}
    Let $(\frakM, \calH, \calH_+,J)$ be a standard form of the von Neumann algebra $\frakM$ and let $(T_t)_{t\ge 0}$ and $(S_t)_{t\ge 0}$ be real and self-adjoint $C_0$-semigroups on the Hilbert space $\calH$ generated by densely defined, closed, symmetric, and accretive sesquilinear forms $\fraka: \dom{\fraka} \times \dom{\fraka} \to \bbC$ and $\frakb: \dom{\frakb} \times \dom{\frakb} \to \bbC$ respectively. 

    If $(T_t)_{t\ge 0}$ and $(S_t)_{t\ge 0}$ are both positive, then the following conditions are equivalent.
    \begin{enumerate}[\upshape (i)]
        \item The semigroup $(T_t)_{t\ge 0}$ is dominated by $(S_t)_{t\ge 0}$, that is, $T_t u\le S_t u$ for all $u\in \calH_+$ and all $t\ge 0$.

        \item The inequality $\fraka\left(\frac{u + (u \wedge v)}{2}\right) + \frakb\left(\frac{v + (u \vee v)}{2}\right) \le \fraka(u)+\frakb(v)$ is true for all $u,v\in \calH_+$.

        \item Each of the following conditions is satisfied:
            \begin{enumerate}[\upshape (a)]
                \item The inclusion $\dom{\fraka}\subseteq \dom{\frakb}$ holds.

                \item If $u\in \dom{\fraka}$ and $v\in \dom{\frakb}$ such that $0\le v\le u$, then $v \in \dom{\fraka}$.

                \item For each $0\le u,v\in \dom{\fraka}$, we have $\fraka(u,v)\ge \frakb(u,v)$.
            \end{enumerate}
    \end{enumerate}
\end{theorem}

\begin{proof}%[Proof of Theorem~\ref{thm:both-positive}]
    First of all, consider the closed convex set
    \[
        C:= \{ (a,b)\in \calH \times \calH: 0\le a\le b\}
    \]
    and the product semigroup on $\calH\times \calH$ given by
    \[
        R_t:= \begin{pmatrix}
                    T_t & 0     \\
                    0    & S_t  
              \end{pmatrix} \qquad (t\ge 0).
    \]
    The semigroup $(R_t)_{t\ge 0}$ is  generated by the form $\frakc$ with $\dom{\frakc}=\dom{\fraka}\times \dom{\frakb}$ and
    \[
        \frakc( (u_0,v_0), (u_1, v_1) ) = \fraka(u_0, u_1)+ \frakb(v_0,v_1).
    \]
    Because the semigroups are positive, the domination of $(T_t)_{t\ge 0}$ by $(S_t)_{t\ge 0}$ is equivalent to  the closed convex set $C$ being invariant under the semigroup $(R_t)_{t\ge 0}$.

    As in the proof of Theorem~\ref{thm:general}, we compute the orthogonal projection $P$ of $\calH\times \calH$ onto $C$. We claim
    \begin{equation}
        \label{eq:projection-positive}
        P(u,v) = \frac12\big(u + (u \wedge v), v+ (u \vee v)\big) \quad \text{whenever }u,v\ge 0.
    \end{equation}
    To prove the claim, let $u,v\in \calH_+$. We of course have  $0\le u + (u \wedge v)\le  v+ (u \vee v)$. Moreover, for each $(a,b)\in C$, we have 
    \begin{align*}
       A&:= \duality{(u,v)-\frac12\left(u + (u \wedge v), v+ (u \vee v)\right)}{(a,b)-\frac12\left(u + (u \wedge v), v+ (u \vee v)\right)} \\
        &=  \duality{\frac{u-(u\wedge v)}{2}}{a-\frac{u + (u \wedge v)}{2}}- \duality{\frac{(u\vee v)-v}{2}}{b-\frac{v + (u \vee v)}{2}}\\
       &=  \duality{\frac{u-(u\wedge v)}{2}}{a-\frac{u + (u \wedge v)}{2}}- \duality{\frac{u-(u\wedge v)}{2}}{b-\frac{v + (u \vee v)}{2}}.
    \end{align*}
    On further simplification, we get
    \begin{align*}
        A &=  \duality{\frac{u-(u\wedge v)}{2}}{a-b+\frac{ (u \vee v) - u + v - (u \wedge v)  }{2}}\\
          &=  \duality{\frac{u-(u\wedge v)}{2}}{a-b+(u \vee v) - u}\\
          %&=  \frac{1}{2}\duality{(v-u)_-}{a-b+(v-u)_+}\\
          &=  \frac{1}{2}\duality{(v-u)_-}{a-b}+\frac12 \duality{(v-u)_-}{(v-u)_+}\\
          &=  \frac{1}{2}\duality{(v-u)_-}{a-b}
    \end{align*}
    where the last equality is obtained using the orthogonality of the positive and negative parts.
    As $(a,b)\in C$, it follows that $A\le 0$. As a consequence, the expression~\eqref{eq:projection-positive} follows due to the criterion~\eqref{eq:projection-criterion}; here we have implicitly used that the cone is self-polar.

   ``(i) $\Leftrightarrow$ (ii)'': As the forms are symmetric, so taking $C_0=\calH_+\times \calH_+$ in Theorem~\ref{thm:barthelemy}, we obtain $(R_t)_{t\ge 0}$ leaves $C$ invariant if and only if $\frakc(P(u,v))\le \frakc( (u,v))$ for all $u,v\in \calH_+$. The equivalence can now be obtained at once from~\eqref{eq:projection-positive}.

    ``(i) $\Rightarrow$ (iii)'': Suppose that $(S_t)_{t\ge 0}$ dominates $(T_t)_{t\ge 0}$.
     This yields, as noted above, the invariance of the closed convex set $C$ under the semigroup $(R_t)_{t\ge 0}$. In particular, $P(\dom{\fraka}\times \dom{\frakb})\subseteq \dom{\fraka}\times \dom{\frakb}$ by Theorem~\ref{thm:barthelemy}. So if $u,v$ are given as in~(b), then using~\eqref{eq:projection-positive}, we get
    \[
        \frac12(u+v,u+v)= P(u,v) \in \dom{\fraka}\times \dom{\frakb}.
    \]
    In particular, $v\in \dom{\fraka}$, as $\dom{\fraka}$ is a subspace. Moreover, using~\eqref{eq:projection-positive}, we get
    \[
        \frac12(u,u)=P(u,0)\in \dom{\fraka}\times \dom{\frakb},
    \]
    for all $0\le u \in \dom{\fraka}$.
    Whence, $\dom{\fraka} \cap \calH_+ \subseteq \dom{\frakb}$.
    Now, let $u\in \dom{\fraka}$. Since $(T_t)_{t\ge 0}$ is positive, the Beurling-Deny criterion \cite[Theorem~2.53]{Ci08} gives  that $ u_+, u_-\in \dom{\fraka}\cap \calH_+ \subseteq \dom{\frakb}$, and consequently $u=u_+-u_-\in \dom{\frakb}$. In fact, positivity of $(T_t)_{t\geq 0}$ even implies that it is real and so, $\dom{\fraka}^J$ is generating for $\dom{\fraka}$.
    We infer that $\dom{\fraka}\subseteq \dom{\frakb}$.
    
    We are left to prove~(c).
    For this purpose, we define the bounded, symmetric, and accretive sesquilinear forms $\fraka^t$  as 
    \[
        \fraka^t(\argument,\argument):= \frac{1}{t}\duality{(I-T_t)\argument}{\argument}
    \]
    and analogously the forms $\frakb^t$  for $t>0$. These forms satisfy 
    \[
        \fraka^t(u,v)-\frakb^t(u,v)= \frac{1}{t} \duality{(S_t-T_t)u}{v} \ge 0
    \]
    for each $0\le u,v\in \dom{\fraka}\subseteq \dom{\frakb}$ and all $t>0$. Letting $t\downarrow 0$ in the above inequality yields~(c); here we have used that $\lim_{t\to 0}\fraka^t(\argument, \argument)= \fraka(\argument,\argument)$ and similarly for $\frakb^t$ (see, for instance, \cite[Lemma~1.56]{Ou04}).

    ``(iii) $\Rightarrow$ (i)'': Let  $(u,v)\in \dom{\fraka}\times\dom{\frakb}$ with $u,v\ge 0$.
    By~(a), we have $u,v\in \dom{\frakb}$.  This allows us to infer from the positivity of the semigroup $(S_t)_{t\ge 0}$ that $u \vee v, u \wedge v \in \dom{\frakb}$ (see \cite[Proposition~4.5 and  Theorem~4.7]{Ci97}).
    On the other hand, $0 \le u \wedge v \le u$, so we can employ~(b) to get $u\wedge v\in \dom{\fraka}$. Wherefore, the equality~\eqref{eq:projection-positive} yields $P(u,v)\in \dom{\fraka}\times \dom{\frakb}$.

    A direct application of~\eqref{eq:projection-positive} and~(c) gives that for every $(u,v)\in \dom{\fraka}\times \dom{\frakb}$ with $u,v\ge 0$,
    \begin{align*}
        \frakc( (u,v), (u,v)-P(u,v) ) &= \fraka \left(u, \frac{u-(u\wedge v)}{2}\right)+\frakb\left(v,\frac{v-(u\vee v)}{2}\right)\\
                                      &\ge \frakb \left(u, \frac{u-(u\wedge v)}{2}\right)+\frakb\left(v,\frac{v-(u\vee v)}{2}\right)\\
                                      &=  \frakd( (u,v), (u,v)-P(u,v) );
    \end{align*}
    where $\frakd$ is the form with domain $\dom{\frakb}\times \dom{\frakb}$ and
    \[
        \frakd( (u_0,v_0), (u_1, v_1) ) := \frakb(u_0, u_1)+ \frakb(v_0,v_1).
    \]
    Of course, the semigroup generated by $\frakd$ on $\calH\times \calH$ is given by
    \[
        U_t:= \begin{pmatrix}
                    S_t & 0     \\
                    0   & S_t  
              \end{pmatrix} \qquad (t\ge 0).
    \]
    Positivity of $(S _t)_{t\ge 0}$ implies that $U_tC\subseteq C$ for $t\ge 0$. Therefore, the characterisation of invariance of closed convex sets Theorem~\ref{thm:barthelemy} yields
    \[
        \frakd( (u,v), (u,v)-P(u,v) )\ge 0.
    \]
    Substituting above, we get $\frakc( (u,v), (u,v)-P(u,v) )\ge 0$. 
    
    Finally, let $C_0$ be the closed convex set $\calH_+\times \calH_+$. We have prove that the inclusion $P(\dom{\frakc}\cap C_0)\subseteq \dom{\frakc}\cap C_0$ and the inequality $\re \frakc( (u,v), (u,v)-P(u,v) )\ge 0$ are satisfied for all $(u,v)\in\dom{\frakc} \cap C_0$. Another application of Theorem~\ref{thm:barthelemy} gives that $R_tC\subseteq C$ for all $t\ge 0$, which proves~(i).
\end{proof}

\begin{example}
    Let $(\frakM, \calH, \calH_+,J)$ be a standard form of the von Neumann algebra $\frakM$. Given a self-adjoint operator $(\calb, \dom{\calb})$ on $\calH$ affiliated to $\frakM$, we let $d_{\calb}$ be the unbounded derivation 
    \[ \dom{d_{\calb}  } : = \dom{\calb} \cap  J \dom{\calb},\qquad d_{\calb} : = i [ a - JaJ];  \]
     see \cite[Section~5]{Ci97} for details. It was shown in \cite[Proposition~5.4]{Ci97}  that the form
    \[
        \dom{\frakb}=\dom{d_{\calb}}, \qquad \frakb(u):= \norm{d_{\calb}u}^2
    \]
    is closable and its closure $\overline{\frakb}$ generates a positive semigroup $(S_t)_{t\ge 0}$. Therefore, if $M:\calH\to \calH$ is a positive bounded operator, then the closure of the form
    \[
        \dom{\fraka}:=\dom{\frakb}, \qquad \fraka(u):= \frakb(u) + \duality{Mu}{u}
    \]
    also generates a positive semigroup, say $(T_t)_{t\ge 0}$. Clearly, Theorem~\ref{thm:both-positive} implies that the semigroup $(S_t)_{t\ge 0}$ dominates $(T_t)_{t\ge 0}$.
\end{example}

\section{Domination of a semigroup by a positive semigroup}
    \label{sec:domination-not-real}    

Every element $u$ of a Hilbert lattice $\calH$ satisfies
\[
    \modulus{u}= \sup\{\re(e^{i\theta}u): \theta \in [0,2\pi]\}.
\]
This allows us to generalise the definition of domination of semigroups to the case when the dominated semigroup is not necessarily real and, in turn, generalise Theorem~\ref{thm:general}.

Let $(\frakM, \calH, \calH_+,J)$ be a standard form of the von Neumann algebra $\frakM$ and let $(T_t)_{t\ge 0}$ and $(S_t)_{t\ge 0}$ be (not necessarily real) self-adjoint $C_0$-semigroups on the Hilbert space $\calH$ generated by closed quadratic forms $\fraka: \calH \to [0,\infty]$ and $\frakb: \calH\to [0,\infty]$ respectively. 

We say that the semigroup $(T_t)_{t\ge 0}$ is \emph{dominated} by $(S_t)_{t\ge 0}$ if
\[
    \left( \re(e^{i\theta} u) \le v \text{ for all }\theta \in [0,2\pi] \right) \Rightarrow \left(\re(e^{i\theta}T_t u)\le S_t v \text{ for all }\theta \in [0,2\pi]\right)
\]
 for all $u,v \in \calH$ and all $t\ge 0$. In this case, of course, $(S_t)_{t\ge 0}$ is automatically positive and, in particular, real.
 In particular, the above domination is equivalent to the following condition:~the semigroup $(R_t)_{t\ge 0}$ given by
 \[
        R_t:= \begin{pmatrix}
                    T_t & 0     \\
                    0    & S_t  
              \end{pmatrix} \qquad (t\ge 0)
\]
leaves the closed convex sets
\[
    C_{\theta}:= \{ (a,b) \in \calH \times \calH : \re(e^{i \theta} a)\le b\} 
\]
invariant for each $\theta \in [0,2\pi]$. As in the preceding sections, one can show that the orthogonal projections $P_{\theta}$ onto the sets $C_{\theta}$ satisfy
\[
    P_{\theta} (u,v) = \left(u-\frac{1}{2} \left(\re(e^{i\theta} u)-v\right)^+, v+\frac{1}{2} \left(\re(e^{i\theta} u)-v\right)^+\right)
\]
for each $(u,v)\in \calH \times \calH^J$ and all $\theta \in [0,2\pi]$.
Now, we may proceed exactly as in Section~\ref{sec:domination-general} (employing Theorem~\ref{thm:barthelemy}) in order to obtain the following characterisation:

\begin{theorem}
    Let $(\frakM, \calH, \calH_+,J)$ be a standard form of the von Neumann algebra $\frakM$ and let $(T_t)_{t\ge 0}$ and $(S_t)_{t\ge 0}$ be self-adjoint $C_0$-semigroups on the Hilbert space $\calH$ generated by closed quadratic forms $\fraka: \calH \to [0,\infty]$ and $\frakb: \calH\to [0,\infty]$ respectively. 

    If the semigroup $(S_t)_{t\ge 0}$ is real, then the following are equivalent.
    \begin{enumerate}[\upshape(i)]
        \item The semigroup $(T_t)_{t\ge 0}$ is dominated by $(S_t)_{t\ge 0}$, that is,
        \[
            \left( \re(e^{i\theta} u) \le v \text{ for all }\theta \in [0,2\pi] \right) \Rightarrow \left(\re(e^{i\theta}T_t u)\le S_t v \text{ for all }\theta \in [0,2\pi]\right)
        \]
        for all $u,v \in \calH$ and all $t\ge 0$. 

        \item The inequality \[\fraka\left( u - \frac{1}{2}\left(\re(e^{i\theta} u-v\right)^+\right) +\frakb\left(v+ \frac{1}{2}\left(\re(e^{i\theta} u-v\right)^+\right) \le \fraka(u)+\frakb(v)\] holds for all $(u,v)\in \calH\times \calH^J$ and for all $\theta \in [0,2\pi]$.
    \end{enumerate}
\end{theorem}

\subsection*{Acknowledgements} 

The first author is grateful to Melchior Wirth for various fruitful discussions about standard forms and for pointing out the reference \cite{Co74}.

\bibliographystyle{plainurl}
\bibliography{ralph}

\end{document}